\documentclass{amsart}

\usepackage[utf8]{inputenc}\usepackage[T1]{fontenc}

\usepackage{amsmath,amsthm,amssymb}
\usepackage{mathtools} 
\usepackage{stmaryrd} 
\usepackage{mathrsfs} 
\usepackage{enumitem}
\usepackage{marginnote}
\newcommand{\Marginpar}[1]{\marginpar{\tiny{#1}}}
\renewcommand{\Marginpar}[1]{}
\usepackage{kbordermatrix,xfrac}
\usepackage[textsize=tiny]{todonotes}

\usepackage{tensor}

\usepackage{hyperref} 

\newtheorem{theorem}{Theorem}[section]
\newtheorem{proposition}[theorem]{Proposition}
\newtheorem{lemma}[theorem]{Lemma}
\newtheorem{corollary}[theorem]{Corollary}

\theoremstyle{definition}

\theoremstyle{remark}
\newtheorem{remark}[theorem]{Remark}

\def\R{\mathbb{R}}\def\C{\mathbb{C}}
\DeclareMathOperator\tr{tr}
\DeclareMathOperator\Pf{Pf}

\DeclareMathOperator\End{End}

\DeclareMathOperator\vol{vol}

\usepackage[english]{babel}
\useshorthands{"}
\defineshorthand{"-}{\nobreakdash-\hspace{0pt}}

\title[Constant curvature almost K\"ahler manifolds]{
Closed almost K\"ahler 4-manifolds of constant non-negative Hermitian holomorphic sectional curvature are K\"ahler.
}

\author{Mehdi Lejmi}
\address{Department of Mathematics, Bronx Community College of CUNY, Bronx, NY 10453, USA.}
\email{mehdi.lejmi@bcc.cuny.edu}

\author{Markus Upmeier}
\thanks{The authors are thankful to Vestislav Apostolov for his comments. The first author was supported in part by a PSC-CUNY award \#60053-00 48, jointly funded by The Professional Staff Congress and The City University of New York.
The first author thanks all the members of the Math Department of University of Augsburg for their warm hospitality during his stay in July 2017.
The second author was partially funded by the Deutsche Forschungsgemeinschaft (DFG, German Research Foundation) -- UP 85/2-1, UP 85/3-1}
\address{Universit\"atsstrasse 14, 86159 Augsburg, Germany.}
\email{Markus.Upmeier@math.uni-augsburg.de}

\begin{document}

\maketitle


\begin{abstract}
We show that a closed almost K\"ahler 4-manifold of globally constant
holomorphic sectional curvature $k\geq 0$ with respect to the canonical Hermitian
connection is automatically K\"ahler. The same result holds for $k<0$ if we require
in addition that the Ricci curvature is $J$-invariant.
The proofs are based on the observation that such manifolds are self-dual,
so that Chern--Weil theory implies useful integral formulas, which are then
combined with results from Seiberg--Witten theory.
\end{abstract}

\section{Introduction}

Dating back to the Goldberg conjecture \cite{MR0238238}, the question
how the geometry of a closed almost K\"ahler $4$-manifold can force the
integrability of an almost complex structure has been considered by many
authors. This conjecture has been verified for non-negative scalar curvature
by Sekigawa \cite{MR787184}.
Much else of what is known has now been subsumed by
Apostolov--Armstrong--Dr\u{a}ghici in \cite{MR1921552},
where it is shown that the third Gray curvature condition is in fact sufficient.
There are also strong results for the $*$-Ricci curvature \cite{MR1604803}, \cite{MR1317012}.
Assuming non-negative scalar curvature, the third Gray condition can 
be relaxed to the Ricci tensor being $J$-invariant \cite{MR1675380}.

Strengthening the Einstein condition, Blair \cite{MR1043404} has shown
that almost K\"ahler manifolds of constant sectional curvature are flat
K\"ahler. For non-flat examples it is natural to consider instead
the holomorphic sectional curvature, where one restricts the
sectional curvature to $J$-invariant planes.

Restricting to a subclass of the Gray--Hervella classification, say almost K\"ahler,
the problem of classifying manifolds of constant holomorphic sectional curvature
was posed by Gray--Vanhecke \cite{MR531928}. A related problem is understanding
when the notions of pointwise constant and globally constant holomorphic sectional
curvature agree. This holds for nearly K\"ahler manifolds \cite{MR0358648}, but
there are non-compact counterexamples in the almost K\"ahler case \cite{MR1806472}.
Moreover, for almost K\"ahler manifolds the classification problem so far has remained
inconclusive (see also~\cite{MR1638199,MR745268}). 

The purpose of this paper is to prove such a classification result for the Hermitian
holomorphic sectional curvature instead of the Riemannian one. We obtain optimal
results for non-negative curvature, while in the negative case we need to impose the
`natural' condition of the Ricci tensor being $J$-invariant.

Note that constant Hermitian holomorphic sectional curvature does not obviously
imply the Einstein condition, nor that any of the scalar curvatures are constant.
Assuming this, we also obtain a partial result in Corollary~\ref{Mehdi-obs}.

\subsection{Overview of results}

Let $(M,g,J,F)$ be an almost Hermitian $4$-manifold. Define the (first canonical) Hermitian connection by
\begin{equation}\label{canonical-connection}
	\nabla_X Y \coloneqq D^g_XY - \frac12 J(D^g_XJ)Y.
\end{equation}
From its curvature we derive the Hermitian holomorphic sectional curvature
\begin{equation}\label{sectional-curvature}
H(X) \coloneqq -R^\nabla_{X,JX,X,JX},
\end{equation}
a function on the unit tangent bundle.
We now state the main results of this paper.

\begin{theorem}\label{main-theorem}
Let $M$ be a closed almost K\"ahler $4$-manifold of globally constant Hermitian holomorphic sectional
curvature $k\geq 0$.

Then $M$ is K\"ahler--Einstein, holomorphically isometric to:
\begin{description}
\item[($k>0$)] $\mathbb{C}P^2$ with the Fubini--Study metric.
\item[($k=0$)] a complex torus or a hyperelliptic curve with the Ricci-flat K\"ahler metric.
\end{description}
\end{theorem}

Recall here that a hyperelliptic curve is a quotient of a complex torus by a finite free group action. This classification is well-known for K\"ahler manifolds (see~\cite[Theorems~7.8,~7.9]{MR1393941} and \cite{MR0052869,MR0063740} in the simply-connected case and \cite[Theorem~2]{MR1348147} in general). The main goal of this paper is to show that $J$ is automatically integrable.

In case $k<0$ we shall prove the following weaker result:

\begin{theorem}\label{main-theorem-negative}
Let $M$ be a closed almost K\"ahler $4$-manifold of pointwise constant Hermitian holomorphic sectional
curvature $k < 0$. Assume also that the Ricci tensor is $J$-invariant. Then $M$ is K\"ahler--Einstein, holomorphically isometric to a compact quotient of the complex hyperbolic ball $\mathbb{B}^4$ with the Bergman metric.
\end{theorem}

The proofs rely on the following pointwise result of independent interest, in which $M$ may be non-compact:

\begin{theorem}\label{thm:char}
Let $M$ be an almost Hermitian $4$-manifold. The holomorphic sectional curvature with respect to the Hermitian connection is constant $k$ at the point $p\in M$ if and only if at that point
\begin{enumerate}
\item
$W^-=0$,
\item
$\ast \rho = r$.
\end{enumerate}
\end{theorem}

Condition ii) may also be expressed using the (Riemannian) Ricci tensor, see Proposition~\ref{Einstein-Kaehler} (we refer to~\cite{MR793346} in the Hermitian case).
Hence in proving Theorems~\ref{main-theorem}, \ref{main-theorem-negative} we may restrict attention to \emph{self-dual manifolds}, meaning $W^-=0$. Their classification is an old and in general still open problem, but under additional assumptions many results have been obtained. See \cite{ MR1956815,MR1348147,MR867684, MR861766} for results and further overview. Our main theorems can also be regarded in this way.

\subsection{Strategy of proof}

The first step is to reformulate constant Hermitian holomorphic sectional curvature in terms of the Riemannian curvature tensor (Theorem~\ref{thm:char}). This an algebraic argument at a point, based on the decomposition of the Riemannian curvature tensor in dimension $4$ and the explicit nature of the gauge potential in \eqref{canonical-connection}. In some sense, the assumption of constant curvature is played off against the symmetries of the Riemannian curvature tensor.  This is carried out in Section~\ref{sec:self-dual}, after having recalled some preliminaries in the next section.

The next step is then in Section~\ref{sec:21} to improve in the almost K\"ahler case our understanding of the Hermitian curvature tensor. It is remarkable that in \eqref{fullChernDecomposition} we obtain information on the full curvature tensor, even though our assumptions depend only upon its $(1,1)$-part.

Up to this point our arguments are mostly algebraic. To proceed, we must exploit consequences of the differential Bianchi identity. Thus in Section~\ref{sec:integral-formulas} we formulate the index theorems for the signature and the Euler characteristic using Chern--Weil theory. Applied to the Levi-Civita and the Hermitian connection, we obtain further information \eqref{Chern-Weil-Conclusion}, \eqref{chi}, \eqref{sigma}.

The formulas are then used in Section~\ref{sec:proofs} to show K\"ahlerness under further topological restrictions. Finally, combined with deep results from Seiberg--Witten theory, these results imply Theorem~\ref{main-theorem} in the case $k\geq 0$. 

Theorem~\ref{main-theorem-negative} ($k<0$) follows by combining our results with formulas for the Bach tensor obtained in \cite{MR1782093}. These formulas require the Ricci tensor to be $J$-invariant.
It is well possible that this additional assumption in Theorem~\ref{main-theorem-negative} may be removed.

\section{Preliminaries}

\subsection{Conventions}

Throughout let $(M,J,g,F)$ be an almost Hermitian $4$-manifold.
Thus ${J\colon TM \to TM}$ is an almost complex structure, $g$ is
a Riemannian metric for which $J$ is orthogonal, and $F=g(J\cdot,\cdot)$.
Later we will also assume $dF=0$ so that we have
an almost K\"ahler structure. Recall that an almost Hermitian manifold is K\"ahler precisely when $J$ is parallel for 
the Levi-Civita connection $D^g J=0$.

Let $(z_1,z_2)$ be a local orthonormal frame of $T^{1,0}M$ for the induced Hermitian metric $h(Z,W)=g_\C(Z,\bar{W})$ on $TM\otimes \C= T^{1,0}M \oplus T^{0,1}M$ split in the usual fashion. Using the dual frame,
 the fundamental form is $F=i(z^1\bar{z}^{\bar{1}} + z^2\bar{z}^{\bar{2}})$.
\Marginpar{$\vol=\frac{F^2}{2}=e^1\wedge Je^1 \wedge e^2\wedge Je^2 = z^{12}\bar{z}^{\overline{12}}$}
All tensors are extended complex linearly and we adopt the summation convention.

\subsection{Two-forms on $4$-manifolds}

The Hodge operator decomposes the bundle of two-forms into the self-dual and anti-self-dual parts
\begin{equation}\label{2formsSDASD}
	\Lambda^2=\Lambda^+ \oplus \Lambda^-.
\end{equation}
For the structure group $U(2) \subset SO(4)$ we may split further
\Marginpar{Also $\Lambda_0^{1,1}\cap \Lambda^2_\R = \Lambda^-$ and similarly for $\Lambda^+$. $\ast z^{1\bar{1}}=z^{2\bar{2}}$, $\ast z^{1\bar{2}}=-z^{1\bar{2}}$}
\begin{align}\label{dual-and-selfdual}
\Lambda^+\otimes \C &= \Lambda^{2,0} \oplus \Lambda^{0,2}\oplus \C\cdot F,
&\Lambda^-\otimes \C &= \Lambda^{1,1}_0.
\end{align}
Here $\Lambda^{1,1}_0$ stands for complex $(1,1)$-forms
pointwise orthogonal to $F$.

\subsection{Gauge potential}

The gauge potential ${A=\nabla - D^g}$ of the canonical connection \eqref{canonical-connection} with respect to the Levi-Civita connection $D^g$ is complex anti-linear
\begin{equation}\label{A-cplx-antilinear}
A_X\circ J = -J\circ A_X.
\end{equation}
In dimension four, the almost K\"ahler condition is equivalent to
\begin{equation}\label{def:21symplectic}
A_{JX}=-J\circ A_X.
\end{equation}
Note that $M$ is K\"ahler $\iff A=0$.



\subsection{Curvature decomposition}
\renewcommand{\kbldelim}{(}\renewcommand{\kbrdelim}{)}
\setlength{\kbcolsep}{-5pt}

Regard the \emph{Hermitian curvature} (and similarly the Riemannian curvature) as a bilinear form on $\Lambda^2$, grouping $XY$ and $ZW$, by
\begin{equation}\label{Rconvention}
	R^\nabla_{XYZW} \coloneqq g([\nabla_X,\nabla_Y] Z - \nabla_{[X,Y]}Z,W).
\end{equation}
When we decompose $\Lambda^2$ into direct summands, we get a corresponding decomposition of $R^\nabla$ into a matrix of bilinear forms, where the first entry corresponds to the rows. The representing matrix of a bilinear form with respect to a basis (one-dimensional summands) will be indicated by `$\equiv$'.

All of the algebraic properties of the \emph{Riemannian} curvature tensor $R^g$ are summarized in the following representation with respect to \eqref{2formsSDASD} and \eqref{dual-and-selfdual}, see \cite{MR1604803,MR867684}:
\begin{equation}\label{LC-decomposition}
\begin{aligned}
-R^g&=
 \kbordermatrix{&\Lambda^+&  &\Lambda^-\\
 & W^+ + \frac{s_g}{12}g & \vrule & R_0\\\cline{2-4}
 &R_0^T&\vrule&W^- + \frac{s_g}{12}g
 }\\
 &=
 \kbordermatrix{&\C F&\Lambda^{2,0}\oplus\Lambda^{0,2}&&\Lambda_0^{1,1}\\
&d\cdot g&W_F^+	&\vrule&	R_F\\
 & (W_F^+)^T&W_{00}^+ + \frac{c}{2}g	&\vrule&	R_{00}\\
\cline{2-5}
&	R_F^T&R_{00}^T 				&\vrule&W^- + \frac{s_g}{12}g
}
\end{aligned}
\end{equation}
%
Here $W^\pm$ are the \emph{Weyl curvatures}, trace-free symmetric bilinear forms on $\Lambda^\pm$ and $s_g$ denotes the Riemannian scalar curvature. The \emph{$*$-scalar curvature} is
\begin{align}\label{def:starScalar}
s_*&\coloneqq 4d,
&\frac{s_g}{4}&=c+d.
\end{align}
Moreover $R_0 \colon \Lambda^+ \otimes \Lambda^- \to \R$ corresponds to the \emph{trace-free Riemannian Ricci tensor} $r_0$, namely $R_0(\cdot)=\frac12 \{r_0,\cdot\}$ is the anti-commutator, see \cite[(A.1.8)]{gauduchon2010calabi}. Finally $R_0^T(x,y)\coloneqq R_0(y,x)$\Marginpar{I.e.~the transpose matrix. In the complexified equations below, it is also the transpose, without conjugation}, 
and $R_{00}, R_F$ denote further restrictions in the first argument. We take the tensorial norm for bilinear forms, even when they are symmetric.

\subsection{Ricci forms}

Having torsion, the curvature tensor of the canonical connection has fewer symmetries than the Riemannian one; for example the algebraic Bianchi identity no longer holds. By contracting indices we now obtain two \emph{Ricci forms}
\begin{equation}\label{def:three-ricci-forms}
\begin{aligned}
\rho &\coloneqq i R\indices{^\nabla_{\alpha\bar\beta\gamma}^\gamma} z^\alpha\wedge \bar{z}^{\bar\beta},\\
r &\coloneqq  iR\indices{^\nabla_\gamma^\gamma_{\lambda\bar\mu}} z^\lambda\wedge\bar{z}^{\bar\mu}.
\end{aligned}
\end{equation}
The \emph{Chern} and \emph{Hermitian scalar curvatures} are obtained by a further trace
\begin{equation}
\begin{aligned}
s_C &\coloneqq \Lambda(\rho) = \Lambda(r) = R\indices{^\nabla_\alpha^\alpha_\gamma^\gamma}\label{sC},  &s_H &\coloneqq 2s_C.\\
\end{aligned}
\end{equation}
In the K\"ahler case, the canonical and the Levi-Civita connection agree so that both forms in \eqref{def:three-ricci-forms} are equal to the usual Ricci form.

\subsection{Holomorphic sectional curvature}

For $Z\in T^{1,0}M$ the holomorphic sectional curvature is defined from the $(1,1)$-part of the curvature as
\begin{equation}\label{sectional-curvature}
H(Z) = \frac{R^\nabla_{Z,\bar{Z},Z,\bar{Z}}}{h(Z,Z)h(Z,Z)}.
\end{equation}

The holomorphic sectional curvature is constant at the point $p\in M$ if \eqref{sectional-curvature} is a constant $k(p)$ for all $Z\in T^{1,0}_pM$. We say it is \emph{pointwise constant} if $H$ is constant at each point of $M$. If the constant $k$ is the same at every point $p\in M$ we speak of \emph{globally constant} holomorphic sectional curvature.

Note that the Hermitian connection is always understood.

\section{Relation to Self-dual Manifolds}\label{sec:self-dual}

In this section we will prove Theorem~\ref{thm:char}.

\subsection{Preparatory lemmas}
Being the complexification of a real tensor, the $(1,1)$-part of the curvature has the following form in the basis $(z_{1\bar{1}}, z_{1\bar{2}}, z_{2\bar{1}}, z_{2\bar{2}})$:
\begin{equation}\label{Rbasis}
R^\nabla|_{\Lambda^{1,1}\otimes \Lambda^{1,1}}\equiv\begin{pmatrix}
k & \bar{a} & a & w\\
\bar{a'} & \bar{x} & \bar{v} & \bar{b}\\
a' & v & x & b\\
u & \bar{b'} & b' & l
\end{pmatrix},\quad k,l, u, w \in \R.
\end{equation}

We first show that the $(1,1)$-part of $R^\nabla$ is automatically restricted further.

\begin{lemma}\label{lem:curv-restrict}
Let $M$ be an almost Hermitian $4$-manifold. With respect to the decomposition $\Lambda^{1,1}=\C F\oplus \Lambda_0^{1,1}$ we have
\begin{equation}\label{curv-restricted}
R^\nabla|_{\Lambda^{1,1}\otimes \Lambda^{1,1}} =
-\kbordermatrix{ & \C F& \Lambda_0^{1,1}\\
&\frac{s_C}{2}g & R_F\\
&R_F^T+\beta_0 & W^- + \frac{s_g}{12}g
}
\end{equation}
for some $\beta_0 \colon \Lambda_0^{1,1} \otimes \C F \to \C$. Moreover $a+b=a'+b'$ and $v=\frac{s_g}{12}$ in \eqref{Rbasis}.
\end{lemma}

\begin{proof}
By direct calculation (or see \cite[p.~25]{MR867684})
\[
R^\nabla_{XYZW} = R^g_{XYZW} + \underset{\alpha\coloneqq}{\underbrace{g((\nabla_X A_Y - \nabla_Y A_X - A_{[X,Y]}) Z, W)}} - \underset{\beta\coloneqq}{\underbrace{ g([A_X, A_Y]Z,W)}}.
\]
By \eqref{A-cplx-antilinear}, $\alpha \in \Lambda^2 \otimes \Lambda^{2,0+0,2}$,
so only $\beta$ will contribute to the restriction of $R^\nabla$. 
Since $[A_X,A_Y]$ is complex-linear, $\beta \in \Lambda^2\otimes \Lambda^{1,1}$.
On $\Lambda^{1,1}$ consider the orthonormal basis
\begin{align}\label{ONB}
\left(\frac{i}{\sqrt{2}}(z_{1\bar{1}}+z_{2\bar{2}}),
z_{1\bar{2}},
z_{2\bar{1}},
\frac{i}{\sqrt{2}}(z_{1\bar{1}}-z_{2\bar{2}})\right).
\end{align}
In dimension $4$, $\beta$ is in fact restricted to $\Lambda^2 \otimes \C F$. Indeed, the explicit formula
\begin{equation}\label{explicit-beta}
	\beta(X,Y,z_\alpha,\bar{z}_{\bar\beta}) 
	= A\indices{_{[X\alpha\gamma}}A\indices{_{Y]\bar\beta}^\gamma}
\end{equation}
shows $\beta_{XY1\bar{2}}=\beta_{XY2\bar{1}}=\beta_{XY1\bar{1}}-\beta_{XY2\bar{2}}=0$ (here `$[\,]$' denotes anti-symmetrization): 
\[
	\beta|_{\Lambda^{1,1}\otimes \Lambda^{1,1}} = \kbordermatrix{
	& \C F & \Lambda_0^{1,1}\\
	&\tilde{\beta}\cdot g & 0\\
	&\beta_0 & 0
	}.
\]
Complexifying and restricting to $\Lambda^{1,1}\otimes\Lambda^{1,1}$ we hence have
\begin{equation}\label{Rcorrection}
	R^\nabla=R^g + \alpha - \beta
	=
	-\kbordermatrix{
	& \C F & \Lambda_0^{1,1}\\
	& d\cdot g_\C & R_F\\
	&R_F^T & W^- + \frac{s_g}{12}g_\C
	}+0
	-
	\kbordermatrix{
	& \C F & \Lambda_0^{1,1}\\
	&\tilde{\beta}\cdot g_\C & 0\\
	&\beta_0 & 0
	}.
\end{equation}
Note that the upper left corner of \eqref{curv-restricted} is $-s_C/2$ by definition \eqref{sC}.
A change of basis shows that the bilinear form \eqref{Rbasis} is represented in the basis \eqref{ONB} by
\begin{equation}\label{Rnewbasis}
R^\nabla|_{\Lambda^{1,1}\otimes \Lambda^{1,1}} \equiv \begin{pmatrix}
\frac{-k-l-u-w}{2} & \frac{i}{\sqrt{2}}(\overline{a+b'}) & \frac{i}{\sqrt{2}}(a+b') & \frac{-k+l-u+w}{2}\\
\frac{i}{\sqrt{2}}(\overline{b+a'}) & \bar{x} & \bar{v} & \frac{i}{\sqrt{2}}(\overline{a'-b})\\
\frac{i}{\sqrt{2}}(b+a') & v & x & \frac{i}{\sqrt{2}}(a'-b)\\
\frac{-k+l+u-w}{2} & \frac{i}{\sqrt{2}}(\overline{a-b'}) & \frac{i}{\sqrt{2}}(a-b') & \frac{-k-l+u+w}{2}
\end{pmatrix}.
\end{equation}
In this basis \eqref{ONB} the inner product on $2$-forms has the matrix
\[
g_\C \equiv \begin{pmatrix}
1\\
&&-1\\
&-1\\
&&&1
\end{pmatrix},
\]
so by comparing \eqref{Rcorrection} and \eqref{Rnewbasis} and using that $W^-+\frac{s_g}{12}g$ is symmetric we get
\[
a+b = a'+ b',\quad
v=\frac{s_g}{12} \in \R.
\]
For later use we note that in this basis the condition $W^-=0$ means that the lower right $3\times3$-submatrix
of \eqref{Rnewbasis} reduces to
\[
	\begin{pmatrix}
&\frac{s_g}{12}\\
\frac{s_g}{12}\\
&&-\frac{s_g}{12}
\end{pmatrix}.
\]
In other words, it means $x=0, a=b', u+2v+w=k+l$.
\end{proof}

\begin{lemma}\label{holLemma-curv}
The holomorphic sectional curvature is constant $k$ at a point if and only if \eqref{Rbasis} reduces at that point to
\begin{equation}\label{ChernR-balas}
R^\nabla|_{\Lambda^{1,1}\otimes \Lambda^{1,1}}\equiv \begin{pmatrix}
k & \bar{a} & a & w\\
-\bar{a} & 0 & \bar{v} & \bar{b}\\
-a & v & 0  & b\\
u & -\bar{b} & -b & k
\end{pmatrix}\quad\text{with}\quad
u+v+\bar{v}+w=2k.
\end{equation}
\end{lemma}

\begin{proof}
Let $Z=xz_1+yz_2$ for arbitrary $x,y\in \C$, and expand both sides of the equation $k\cdot h(Z,Z)^2 = R^\nabla(Z,\bar{Z},Z,\bar{Z})$. We have
\[
	k\cdot h(Z,Z)^2 = k|x|^4+2k|x|^2|y|^2 + k|y|^4
\]
while for the right hand side
\begin{align*}
&|x|^4 R_{z_1\bar{z}_1z_1\bar{z}_1}+ |y|^4 R_{z_2\bar{z}_2z_2\bar{z}_2}\\
&+|x|^2|y|^2 \left(R_{z_1\bar{z}_1z_2\bar{z}_2}+R_{z_2\bar{z}_2z_1\bar{z}_1}+R_{z_2\bar{z}_1z_1\bar{z}_2}+R_{z_1\bar{z}_2z_2\bar{z}_1}\right)\\
&+x^2\bar{y}^2 R_{z_1\bar{z}_2z_1\bar{z}_2}
+\bar{x}^2y^2 R_{z_2\bar{z}_1z_2\bar{z}_1}\\
&+x^2\bar{x}\bar{y}\left( R_{z_1\bar{z}_1z_1\bar{z}_2}+R_{z_1\bar{z}_2z_1\bar{z}_1} \right)\\
&+\bar{x}^2xy\left( R_{z_2\bar{z}_1z_1\bar{z}_1}+R_{z_1\bar{z}_1z_2\bar{z}_1} \right)\\
&+ \bar{x}y^2\bar{y} \left( R_{z_2\bar{z}_2z_2\bar{z}_1}+R_{z_2\bar{z}_1z_2\bar{z}_2} \right)\\
&+ x\bar{y}^2y \left( R_{z_2\bar{z}_2z_1\bar{z}_2}+ R_{z_1\bar{z}_2z_2\bar{z}_2} \right)
\end{align*}
Since $x,y \in \C$ are arbitrary, this an equality between polynomials in the variables $x,\bar{x},y,\bar{y}$. For them to agree, all coefficients must be equal.
\end{proof}

\begin{remark}
The above is a simplified proof of a theorem of Balas \cite{MR779217} for Hermitian manifolds.
\end{remark}

\subsection{Proof of Theorem~\textup{\ref{thm:char}}}

First note that by Lemma~\ref{lem:curv-restrict} we automatically have $a+b=a'+b'$ and $v=\frac{s_g}{12}\in \R$ in \eqref{Rbasis}.

Lemma~\ref{holLemma-curv} shows that $H$ is constant $k$ at a point if and only if
\begin{equation}\label{eq:holk1}
	x=0,\quad
	a'=-a,\quad
	a=-b,\quad
	k=l,\quad
	u+2v+w=2k.
\end{equation}
On the other hand, by \eqref{Rnewbasis} and the following discussion, self-duality means\Marginpar{Then also $a'=b$}
\begin{equation}\label{eq:holk2}
	x=0,\quad
	a=b',\quad
	u+2v+w=k+l.
\end{equation}
From \eqref{Rbasis} we read off
\begin{align*}
\rho &= i(k+w)z^{1\bar{1}} + i\overline{(a'+b)}z^{1\bar{2}} + i(a'+b)z^{2\bar{1}} + i(u+l)z^{2\bar{2}}\\
r &= i(k+u)z^{1\bar{1}} + i\overline{(a+b')}z^{1\bar{2}} + i(a+b')z^{2\bar{1}}+i(w+l)z^{2\bar{2}}
\end{align*}
Therefore $\ast \rho = r$ means
\begin{equation}\label{forms-dual}
k=l,\quad
a'+b=-(a+b').
\end{equation}
Under the general assumption $a+b=a'+b'$ and $v\in \R$ it is easy to verify
that \eqref{eq:holk1} is equivalent to \eqref{eq:holk2} with \eqref{forms-dual}.
\hfill\ensuremath{\square}

\section{Sharper results for almost K\"ahler manifolds}\label{sec:21}

We now obtain more information on the terms in
\begin{equation}
R^\nabla=R^g + \alpha - \beta
\end{equation}

\subsection{Full description of $\beta$}

\begin{lemma}
Let $M$ be an almost K\"ahler $4$-manifold. Then
\begin{align}\label{betanorms}
\tilde{\beta}&=-\frac12|A|^2,
&\|\beta_0\|^2&=\frac14|A|^4.
\end{align}
\end{lemma}

Here we use the convention $|A|^2 = \frac12 \sum_{i=1}^4 \tr\left( A_{e_i}^T A_{e_i} \right)$.

\begin{proof}
More precisely, from \eqref{def:21symplectic} we have $\beta \in \Lambda^{1,1}\otimes \C\cdot F$.
Using \eqref{explicit-beta} one sees
\begin{equation}\label{eqn:betamatrix}
	\beta=
	\kbordermatrix{&\C F&\Lambda^{2,0}\oplus\Lambda^{0,2}&&\Lambda_0^{1,1}\\
&\tilde\beta &0	&\vrule&	0\\
& 0&0 &\vrule&	0\\
\cline{2-5}
&\beta_0&0&\vrule&0
}
\end{equation}
where
\begin{align*}
	\tilde\beta &= -|A_{112}|^2-|A_{212}|^2,
	&\beta_0 &=	\begin{pmatrix}
	i\sqrt{2}\cdot A_{112}\overline{A_{212}}\\
	i\sqrt{2}\cdot A_{212}\overline{A_{112}}\\
	|A_{212}|^2-|A_{112}|^2
	\end{pmatrix}.
\end{align*}
In the upper left corner of \eqref{Rcorrection} we recover the well-known formula (\cite{MR1782093}, \cite[(9.4.5)]{gauduchon2010calabi})
\begin{equation}\label{SstarSHerm}
\frac{s_* - s_H}{4} = \frac12|A|^2.\qedhere
\end{equation}
\end{proof}

\begin{proposition}\label{Einstein-Kaehler}
Let $M$ be a self-dual almost Hermitian $4$-manifold. Then $M$ has constant holomorphic
sectional curvature at $p\in M$ precisely when
\begin{equation}\label{RFbeta0}
	R_F=-\frac12 \beta_0^T.
\end{equation}
at that point. In particular, when $M$ is almost K\"ahler and $M$ has pointwise constant holomorphic sectional curvature,
then: $M$ is K\"ahler $\iff R_F=0$.
\end{proposition}

\begin{proof}
Putting the self-duality condition \eqref{eq:holk2} from the proof of Theorem~\ref{thm:char} into \eqref{Rnewbasis} we get representative matrices
\[
	-R_F\equiv (\sqrt{2}i\bar{a}, \sqrt{2} ia, v+w-k),\quad
	R_F^T + \beta_0 \equiv -\begin{pmatrix}
	\sqrt{2}i\bar{b}\\
	\sqrt{2}ib\\
	\frac{u+v-k}{2}
	\end{pmatrix}
\]
with respect to the orthonormal basis \eqref{ONB}. Hence a self-dual manifold has constant holomorphic sectional curvature \eqref{eq:holk1} precisely when \eqref{RFbeta0} holds.

In the almost K\"ahler case we may use \eqref{betanorms} to get
\begin{equation}\label{RF2}
\|R_F\|^2 = \frac14 \|\beta_0\|^2 = \frac{|A|^4}{16}.\qedhere
\end{equation}
\end{proof}

\subsection{Constant holomorphic sectional curvature}

\begin{proposition}
Let $M$ be almost K\"ahler of constant holomorphic sectional curvature $k$. Then we have
\begin{equation}\label{fullChernDecomposition}
-R^\nabla= \kbordermatrix{&\C F&\Lambda^{2,0}\oplus\Lambda^{0,2}&&\Lambda_0^{1,1}\\
&\frac{s_C}{2} g&0&\vrule& R_F\\
&W_F^+&0&\vrule& R_{00}\\
\cline{2-5}
&-R_F^T&0&\vrule& \frac{s_g}{12}g
}
\end{equation}
\end{proposition}
\begin{proof}
Put all the facts
$\alpha \in \Lambda^2\otimes \Lambda^{2,0+0,2}$,
$R^\nabla \in \Lambda^2\otimes \Lambda^{1,1}$,
and \eqref{eqn:betamatrix}, \eqref{RFbeta0} into the
formula $R^\nabla=R^g + \alpha - \beta$.
\end{proof}

We now collect some formulas that will be useful. For these, assume
that $M$ is almost K\"ahler and has pointwise constant holomorphic sectional curvature $k$.

Comparing the upper left corners of \eqref{Rnewbasis} and \eqref{fullChernDecomposition} and then using \eqref{eq:holk1} we see
\begin{equation}\label{sCv}
s_C = 4k-2v.
\end{equation}
Recall also
\begin{equation}\label{sgv}
s_g = 12v.
\end{equation}
Since $M$ is almost K\"ahler we have for the $*$-scalar curvature \cite[(9.4.5)]{gauduchon2010calabi}
\begin{equation}\label{sStarV}
	s_* = 4s_C - s_g = 16k-20v.
\end{equation}
Moreover \eqref{RFbeta0} can now be written
\begin{equation}\label{RFv}
|R_F|^2 = \frac14\left( \frac{s_* - s_H}{4} \right)^2 = (k-2v)^2
\end{equation}

Putting these formulas into \eqref{SstarSHerm} shows:

\begin{lemma}\label{integrability-criterion}
Let $M$ be almost K\"ahler of pointwise constant holomorphic sectional curvature $k$. We then have $v\leq \frac{k}{2}$ with equality if and only if $M$ is K\"ahler.\hfill\ensuremath{\square}
\end{lemma}

The formulas also show that when $M$ has globally constant holomorphic sectional curvature, the constancy of any of $s_g, s_*, s_C$ is equivalent to that of $v$.

\section{Integral formulas}\label{sec:integral-formulas}

Having understood the pointwise (algebraic) implications of constant holomorphic sectional curvature, we now turn to properties that do not hold for general algebraic curvature tensors. Thus we formulate the consequences of Chern--Weil theory, which stem ultimately from the differential Bianchi identity.

We will assume in this section that $M$ is an almost K\"ahler $4$-manifold of pointwise constant holomorphic sectional curvature, but generalizations are possible.

\subsection{Chern--Weil theory}
Given an arbitrary metric connection $\nabla$ on $TM$ and a polynomial $P$ on $\mathfrak{so}(4)$, invariant under the adjoint action of $SO(4)$, one obtains a differential form $P(R^\nabla)$ by substituting the indeterminants by the curvature $R^\nabla\colon \Lambda^2 \to \mathfrak{so}(TM)$. The upshot of Chern--Weil theory (see for example~\cite{MR0440554}) is that $P(R^\nabla)$ defines a closed form whose cohomology class is independent of $\nabla$. In particular, the integral over $M$ remains the same for all connections.

In the $4$-dimensional case, it suffices to consider the Pontrjagin and Pfaffian polynomials. To express these conveniently, use the metric to identify $\mathfrak{so}(TM)\cong \Lambda^2$ and
decompose as above
\[
	-R^\nabla = \kbordermatrix{&\Lambda^+&  &\Lambda^-\\
 & R_{sd}^+& \vrule & R_{asd}^+\\\cline{2-4}
 & R_{sd}^- & \vrule & R_{asd}^-
 }.
\]
Then
\begin{equation}\label{char-forms}
\begin{aligned}
p_1(R^\nabla) &= \frac{1}{4\pi^2} \left(\|R_{sd}^+\|^2+\|R_{sd}^-\|^2 - \|R_{asd}^+\|^2 - \|R_{asd}^-\|^2 \right)\vol,\\
\Pf(R^\nabla) &= \frac{1}{8\pi^2}\left(\|R_{sd}^+\|^2-\|R_{sd}^-\|^2 - \|R_{asd}^+\|^2 + \|R_{asd}^-\|^2 \right)\vol,
\end{aligned}
\end{equation}
where the norm is induced by the usual inner product $\tr(f^*g)$ on $\End(\Lambda^2)$.
For the Levi-Civita connection we evaluate \eqref{char-forms} using \eqref{LC-decomposition}, $W^-=0$, and $\|I\|^2=3$:
\begin{equation}\label{p1Dg}
\begin{aligned}
p_1(D^{g}) &=  \frac{1}{4\pi^2}\left( \|W^+ +\frac{s_g}{12}I\|^2 + \|R_0^*\|^2 - \|R_0\|^2 - \|W^- + \frac{s_g}{12}I\|^2\right)\operatorname{vol}_g\\
&= \frac{1}{4\pi^2}(\|W^+\|^2 - \|W^-\|^2)\operatorname{vol}_g\\
&= \frac{1}{4\pi^2}\|W^+\|^2\operatorname{vol}_g\\
\end{aligned}
\end{equation}
and\footnote{Traditionally one writes $2\|R_0\|^2=\frac12 \|r_0\|^2$ in terms of the Ricci tensor $r_0$.}
\begin{equation}\label{Pfaffian-LC}
\begin{aligned}
\operatorname{Pf}(D^g) &= \frac{1}{8\pi^2}(\|W^+ +\frac{s_g}{12}I\|^2 + \|W^- + \frac{s_g}{12}I\|^2 - \|R_0^*\|^2 - \|R_0\|^2)\operatorname{vol}_g\\
&=\frac{1}{8\pi^2} (\|W^+\|^2 + \|W^-\|^2 + \frac{1}{24}s_g^2 - 2\|R_0\|^2)\operatorname{vol}_g\\
&= \frac{1}{8\pi^2} (\|W^+\|^2 + \frac{1}{24}s_g^2 - 2\|R_0\|^2)\operatorname{vol}_g
\end{aligned}
\end{equation}
Similarly, for the Hermitian connection we evaluate \eqref{char-forms} using \eqref{fullChernDecomposition}:
\begin{align}
\label{p1Nabla}p_1(\nabla) &= \frac{1}{4\pi^2} \left( \frac{s_C^2}{4} + \|W_F^+\|^2 + \|R_{00}\|^2 - \frac{s_g^2}{48} \right)\operatorname{vol}_g\\
\operatorname{Pf}(\nabla) &= \frac{1}{8\pi^2} \left(\frac{s_C^2}{4}+\|W_F^+\|^2 + \frac{s_g^2}{48} - 2\|R_F\|^2 - \|R_{00}\|^2 \right)\operatorname{vol}_g
\end{align}

\subsection{Index theorems}
The pre-factors in \eqref{char-forms} are chosen so that for the signature and Euler characteristic the classical index theorems hold (see~\cite{MR0440554}):
\begin{equation}\label{indextheorem}
\begin{aligned}
\sigma &= \frac13 \int_M p_1(R^\nabla)\\
\chi &= \int_M \Pf(R^\nabla)
\end{aligned}
\end{equation}

This gives us two expressions for the signature and for the Euler characteristic \eqref{indextheorem}. Equating these leads to the same conclusion in both cases:

\begin{proposition}
Let $M$ be a closed almost K\"ahler $4$-manifold of pointwise constant holomorphic sectional curvature $k$. Then
\textup(we omit the volume form\textup)
\begin{equation}\label{Chern-Weil-Conclusion}
\int_M |W_F^+|^2 + |W_{00}^+|^2 + 4(5k-7v)(k-2v) = \int_M |R_{00}|^2
\end{equation}
\end{proposition}

\begin{proof}
By the above remarks
\[
	\int_M p_1(D^g)= \int_M p_1(\nabla)
\]
and we insert \eqref{p1Dg}, \eqref{p1Nabla}. Then using \eqref{LC-decomposition} we get
\[
\|W^+\|^2 = 2\|W_F^+\|^2 + \|W_{00}^+\|^2 + \frac{1}{6}(3s_C-s_g)^2.
\]
since in the almost K\"ahler case $s_*=4d$, $\frac{s_g}{4}=c+d$, and $\frac{s_g+s_*}{2}=2s_C$. Finally insert \eqref{sCv}--\eqref{sStarV} to obtain the conclusion \eqref{Chern-Weil-Conclusion}. Doing the same computation for the Pfaffian, use $\|R_0\|^2 = \|R_{00}\|^2 + \|R_F\|^2$. This leads to the same formula \eqref{Chern-Weil-Conclusion}.
\end{proof}

Putting \eqref{Chern-Weil-Conclusion} and \eqref{sCv}--\eqref{RFv} into \eqref{indextheorem} gives the following:

\begin{proposition}
Let $M$ be a closed almost K\"ahler $4$-manifold of pointwise constant holomorphic sectional curvature $k$. Then
\begin{align}\label{chi}
\chi &= \frac{-1}{8\pi^2} \int_M |W_{00}^+|^2+(60v^2-72 kv + 18k^2)\\
\frac{3}{2}\sigma &= \frac{1}{8\pi^2} \int_M 2|W_F^+|^2 + |W_{00}^+|^2 + 6(2k-3v)^2 \enskip\geq\enskip 0\label{sigma}
\end{align}
\end{proposition}

Combined with \cite[Lemma~3]{MR1604803} we conclude:

\begin{corollary}\label{Mehdi-obs}
Let $M$ be closed almost K\"ahler of globally constant holomorphic sectional curvature $k$.
Suppose $M$ is simply connected \textup(or, more generally, that $5\chi + 6\sigma \neq 0$\textup) and
that any of $s_g, s_*, s_C$ is constant. Then $M$ is K\"ahler.
\end{corollary}

\begin{proof}
Suppose by contradiction that $J$ is not integrable, so that $v<\frac{k}{2}$ at some point, by Lemma~\ref{integrability-criterion}. Then \eqref{RF2} and \eqref{RFv} show that $|N|^2$ is a non-zero constant.
Hence by \cite[Lemma~3]{MR1604803} we must have
\[
	5\chi + 6\sigma = 0.
\]
This contradicts $\chi=2+b_2 > 0$ and $\sigma\geq 0$ from \eqref{sigma}.
\end{proof}

\section{Proof of Main Theorems}\label{sec:proofs}

\subsection{Intermediate results}

Before proving Theorem~\ref{main-theorem} we need to establish two
preliminary results that give K\"ahlerness under topological restrictions.

\begin{proposition}\label{flatcase}
Let $M$ be closed almost K\"ahler $4$-manifold of pointwise constant holomorphic sectional curvature $k$. Suppose $\sigma=0$ for the signature.

Then $k=0$ and $M$ is K\"ahler, with a Ricci-flat metric.
\end{proposition}

\begin{proof}
By \eqref{sigma}, $\sigma=0$ implies $W_F^+=0$, $W_{00}^+ = 0$, and $v=\frac23 k$. Putting this into \eqref{Chern-Weil-Conclusion} shows $R_{00}=0$ and $k=0$, since the left hand side reduces to $-\frac49 k^2\cdot\operatorname{Vol}(M)$ and the right hand side is non-negative. Hence integrability follows from Lemma~\ref{integrability-criterion}. 

Then from \eqref{RFv} we also get $R_F=0$, so $R_0=R_{00}+R_F=0$. Hence the metric is Ricci-flat.
\end{proof}

We next show a `reverse' Bogomolov--Miyaoka--Yau inequality.

\begin{proposition}\label{thm-sigma-chi}
If $M$ is closed almost K\"ahler of globally constant holomorphic sectional curvature $k \geq 0$ then
\begin{equation}\label{sigma-chi}
	3\sigma \geq \chi.
\end{equation}
Equality holds if and only if $M$ is K\"ahler \textup(even K\"ahler--Einstein\textup).
\end{proposition}

\begin{proof}
Estimate \eqref{sigma-chi} follows by combining \eqref{chi}, \eqref{sigma} to get
\begin{equation}\label{nice1}
	3 \sigma - \chi = \frac{1}{8\pi^2}\int_M |W_F^+|^2 + 3|R_{00}|^2 + 6k(k-2v)
\end{equation}
and the fact $v\leq \frac{k}{2}$. Putting $3\sigma=\chi$
into \eqref{nice1} shows
\begin{equation}\label{temp12049u12}
	0 = \int_M |W_F^+|^2 + 3|R_{00}|^2 + 6k(k-2v)
\end{equation}
which is the sum of three non-negative terms. Hence all summands vanish. When $k\neq 0$ is a global non-zero constant we get
\[
	\int_M (k-2v)\vol = 0.
\]
Since $k-2v\geq 0$ this means $k=2v$ everywhere, hence integrability by Lemma~\ref{integrability-criterion}. If on the other hand we suppose $k=0$ then \eqref{temp12049u12} gives $W_F^+ = 0$, $R_{00}=0$. Putting this into \eqref{Chern-Weil-Conclusion} shows
\[
	0 = \int_M |W_{00}^+|^2 + 56v^2
\]
and so $v=0$. We again therefore have $v=\frac{k}{2}=0$ and may apply Lemma~\ref{integrability-criterion}.
\end{proof}

\subsection{Proof of Theorem~\ref{main-theorem}}

By Lemma~\ref{integrability-criterion} we know that if $v=\frac{k}{2}$ everywhere, then $M$ is K\"ahler. Let us argue by contradiction and therefore suppose that $v<\frac{k}{2}$ somewhere (the case $k>0$ can also be deduced directly). Then
\begin{equation}
\int_M c_1(TM) \cup \omega = \int_M \frac{s_C}{2\pi} = \int_M \frac{4k-2v}{2\pi} = \underbrace{\int_M \frac{3k}{2\pi}}_{\geq 0} + \int_M \frac{k-2v}{2\pi} > 0.
\end{equation}
According to results in Seiberg--Witten theory of Taubes~\cite{MR1324704}, Lalonde--McDuff~\cite{MR1432456}, and Liu~\cite{MR1418572}, this implies that $M$ is symplectomorphic to a ruled surface or to $\C P^2$ (see LeBrun~\cite[Section~2]{MR3358072} for an overview, or also \cite[Theorem~1.2]{MR1417784}). In the case $\C P^2$ we get a contradiction to Proposition \ref{thm-sigma-chi}. Suppose therefore that $M$ is a ruled surface. Since $\sigma \leq 0$ for ruled surfaces, by \eqref{sigma} we must have $\sigma=0$, contradicting Proposition~\ref{flatcase}.

Hence we have shown that $M$ is K\"ahler. The classification stated in Theorem~\ref{main-theorem} now follows from \cite[Theorem~2]{MR1348147}.
\hfill\ensuremath{\square}

\subsection{Proof of Theorem~\ref{main-theorem-negative}}

By Theorem~\ref{thm:char} we have $W^-=0$. Hence the Bach tensor vanishes and so \cite[Remark~2, p.~13]{MR1782093} applies
\begin{equation}\label{conclusion-AD}
	0=\int_M \left( \frac{|ds_g|^2}{2}-s_g\cdot |r_0|^2 \right)\vol.
\end{equation}
The proof in \cite{MR1782093} is a consequence of a Weitzenb\"ock formula and the $J$-invariance of the Ricci tensor is used in a crucial way.

Since $s_g=12v \leq 6k < 0$ we have in \eqref{conclusion-AD} two non-negative terms. Hence $s_g$ is constant and $R_0=0$. Therefore $M$ is K\"ahler--Einstein by Proposition~\ref{Einstein-Kaehler}.\hfill\ensuremath{\square}

\subsection{Further Discussion}

Besides removing in Theorem~\ref{main-theorem-negative} the condition that the Ricci curvature is $J$-invariant, we mention the following open problems:
\begin{enumerate}
\item
Analogous results in the non-compact case. The case of Lie groups will be the topic of an upcoming paper of L.~Vezzoni and the first named author~\cite{Lej-Vez}.
\item
Counterexamples to Schur's Theorem. Are there almost K\"ahler manifolds of pointwise
constant \emph{Hermitian} holomorphic sectional curvature that are not globally constant?
In particular, are there compact examples?
\item
There may be an alternative approach to our results using the twistor space of $M$. Note that in our
situation, the tautological almost complex structure on the twistor space is integrable. What are the further
geometric implications of constant holomorphic sectional curvature in terms of the twistor space?
\end{enumerate}

\bibliographystyle{abbrv}

\bibliography{biblio-OWR}

\end{document}